\documentclass[12pt]{article}
\usepackage{graphicx}
\usepackage{amssymb,latexsym,amscd,amsmath,amsfonts,amsthm,pb-diagram,enumerate}
\usepackage[mathscr]{eucal}
\numberwithin{equation}{section}
\newtheorem{theorem}{Theorem}[section]

\newtheorem{corollary}[theorem]{Corollary}
\newtheorem{lemma}[theorem]{Lemma}

\newtheorem{remark}[theorem]{Remark}

\begin{document}
\begin{center} {\Large Localization at countably infinitely many\\ prime
ideals and applications}
 \footnote [1]{\noindent{\bf Key words and phrases:} Localization; Local cohomology; Associated prime ideal.\\
\indent{\bf AMS Classification 2010:} 13B30; 13D45; 13E99.\\
The second author is supported by Vietnam National Foundation for Science
and Technology Development (NAFOSTED) under grant number 101.04-2015.25 and by the research project No. DHFPT/2015/01 granted by FPT University.
}\\

\end{center}
\begin{center}
KAMAL BAHMANPOUR and PHAM HUNG QUY
\end{center}
\begin{abstract} In this paper we present a technical lemma about localization at countably infinitely many prime ideals. We apply this lemma to get many results about the finiteness of associated prime ideals of local cohomology modules.
\end{abstract}
\section{Introduction}
In this paper, let $R$ be a commutative Noetherian ring. Localization is one of the most important tools in Commutative algebra. Notice that for any multiplicative subset $S$ of $R$, the canonical extension $R \to R_S$ is flat, and many problems in Commutative algebra have good behavior under flat extensions. %Let $\frak p \in \mathrm{Spec}(R)$ be a prime ideal. By localization at $\frak p$ we can reduce a global problem in $R$ to a local problem in $R_{\frak p}$. We will exhibit many advantages of the local ring $R_{\frak p}$ such as the Nakayama lemma, the Krull intersection theorem, etc.
For a set of finitely many prime ideals $\{\frak p_1, ..., \frak p_k \}$ with no containment relations, set $S = R \setminus \cup_{i=1}^k \frak p_i$, we have $R_S$ is a semilocal ring and $\mathrm{Max}(R_S) = \{\frak p_1R_S, ..., \frak p_kR_S\}$. This fact follows from the well known prime avoidance lemma. This statement is false for countably infinitely many prime ideals $\{\frak p_i\}_{i \ge 1}$. For example, let $R = \mathbb{Q}[X, Y]$ and $\{\frak p_i\}_{i \in I}$ is the set of prime ideals of height one. Since $R$ is UFD we have a prime ideal of height one is principal. Moreover $R$ is a countable set, so the set $\{\frak p_i\}_{i \in I}$ is countable. On the other hand every non-constant polynomial must be contained in a prime ideal of height one. Thus $S = R \setminus \cup_{i \in I}\frak p_i = \mathbb{Q}$ and so $R_S = R$. This paper is devoted to the localization at countably infinitely many prime ideals after passing to a certain flat extension. Concretely, we prove the following result.
\begin{lemma}\label{L1.1}
  Let $R$ be a commutative Noetherian ring and $\{\frak p_i\}_{i \ge 1}$ a countable set of prime ideals of $R$ with no containment relation. Consider the formal power series ring $R[[X]]$ and set $S = R[[X]] \setminus \cup_{i \ge 1} \frak p_iR[[X]]$ and $T = R[[X]]_S$. Then $R \to T$ is a flat extension and $\mathrm{Max}(T) = \{\frak p_iT\}_{i \ge 1}$.
\end{lemma}
The above lemma will be proved in the next section. In Section 3 we apply Lemma \ref{L1.1} to get many results about the the finiteness of associated prime ideals of local cohomology modules. Among them, is the following:
\begin{theorem} Let $I$ be an ideal of $R$ and $M$ a finitely generated $R$-module. Then for every $i \ge 0$ the set $\{ \frak p \in \mathrm{Ass}_RH^i_I(M)\,\,:\,\, \mathrm{ht}(\frak p/I) \le 1 \} $ is finite.
  \end{theorem}
Recall that, for any ideal $I$ of $R$ and any $R$-module $M$,
the $i^{th}$ local cohomology module of $M$ with respect to $I$ is
defined as$$H^i_I(M) = \underset{n\geq1} {\varinjlim}\,\,
\text{Ext}^i_R(R/I^n, M).$$We refer the reader to \cite{BS98} or
\cite{G66} for more details about local cohomology.

\section{Localization at countably infinitely many prime
ideals}
 We start this section with the well known result, countable prime avoidance lemma (see \cite[Lemma 13.2]{LW12}).
\begin{lemma}\label{L2.1}
Let $A$ be a Noetherian ring satisfying either of these conditions:
\begin{enumerate}[{(i)}]\rm
\item {\it $A$ is a complete local ring.}
\item {\it There is an uncountable set of elements $\{ \mu _{\lambda}\}_{\lambda} \in \Lambda$ such that $\mu_{\lambda} - \mu_{\gamma}$ is a unit of $A$ for every $\lambda \neq \gamma$.}
\end{enumerate}
Let $\{\frak p_i\}_{i \ge 1}$ a countable set of prime ideals of $A$ and $I$ an ideal such that $I \subseteq \cup_{i \ge 1} \frak p_i$. Then $I \subseteq \frak p_i$ for some $i$.
\end{lemma}
The following technical lemma is the main result of this section.
\begin{lemma}\label{L2.2}
Let $R$ be a commutative Noetherian ring and $\{\frak p_i\}_{i \ge 1}$ a countable set of prime ideals of $R$ with no containment relation. Consider the formal power series ring $R[[X]]$ and set $S = R[[X]] \setminus \cup_{i \ge 1} \frak p_iR[[X]]$ and $T = R[[X]]_S$. Then $R \to T$ is a flat extension and $\mathrm{Max}(T) = \{\frak p_iT\}_{i \ge 1}$.
\end{lemma}
\begin{proof} It is clear that $R \to T$ is flat and $\frak p_iT \in \mathrm{Spec}(T)$ for all $i \ge 1$. We prove that $T$ satisfies the condition (ii) of Lemma \ref{L2.1}. We consider the following subset of elements in $R[[X]]$
  $$\mathcal{B} := \{ \mu_{\lambda} = b_0 + b_1X + \cdots + b_nX^n + \cdots \,\,:\,\, b_i = 0 \text{ or } 1 \text{ and } \mu_{\lambda} \neq 0\}.$$
Since $R[[X]]$ is a subring of $T$, so $\mathcal{B} \subseteq T$. It is clear that $\mathcal{B}$ is an uncountable set. For every $\mu_{\lambda} \neq \mu_{\gamma}$ pair of distinct elements of $\mathcal{B}$ we have
  $$\mu_{\lambda} - \mu_{\gamma} = a_0 + a_1X + \cdots + a_nX^n + \cdots$$
  with $a_i = 0, 1$ or $-1$ and at least one $a_i \neq 0$. Let $k$ be the least integer such that $a_k \neq 0$. We have
 $$\mu_{\lambda} - \mu_{\gamma} = X^k(1 + a_{k+1}X + \cdots )$$
or
 $$\mu_{\lambda} - \mu_{\gamma} = X^k(-1 + a_{k+1}X + \cdots ).$$
We have both $1 + a_{k+1}X + \cdots$ and $-1 + a_{k+1}X + \cdots$ are units in $R[[X]]$ and so are in $T$. Since $X \notin \frak p_iT$ for all $i \ge 1$ we have $X \in S$. Thus $X$ is a unit in $T$. Therefore $\mu_{\lambda} - \mu_{\gamma}$ is a unit in $T$ for every $\mu_{\lambda} \neq \mu_{\gamma}$. Hence $T$ satisfies the countable prime avoidance lemma. Since $\cup_{i \ge 1}\frak p_iT$ is the set of non-units of $T$, we have $I \subseteq \cup_{i \ge 1}\frak p_iT$ for every proper ideal $I$ of $T$. By the countable prime avoidance lemma we have $I \subseteq \frak p_iT$ for some $i$. Therefore $\mathrm{Max}(T) = \{\frak p_iT \}_{i \ge 1}$. The proof is complete.
\end{proof}
\section{Applications}
In this section, let $I$ be an ideal of $R$ and $M$ a finitely generated $R$-module. In general the $i^{\mathrm{th}}$ local cohomology module $H_I^i(M)$ is not finitely generated. Grothendieck asked the following question: Is $\mathrm{Hom}(R/I, H_I^i(M))$ finitely generated for all $i \ge 0$? The first counterexample was given by Hartshorne in \cite{Ha70}. In this paper he introduced the notion of $I$-cofinite modules. An $R$-module $L$ is called $I$-cofinite if $\mathrm{Supp}(L) \subseteq V(I)$ and $\mathrm{Ext}^i_R(R/I, L)$ is finitely generated for all $i \ge 0$. Hartshorne proved that $H_I^i(M)$ is $I$-cofinite for all $i\ge 0$ if $R$ is a complete regular local ring and $I$ is a prime ideal of dimension one. Hartshorne's result was extended by many authors (see, \cite{BN09}, \cite{DM97}, \cite{HK91}, \cite{Y97}). In \cite[Theorem 1.1]{BN09} Bahmanpour and Naghipour proved the following result (see also \cite[Theorem 2.10]{Me12}).
\begin{lemma}\label{L3.1} Let $I$ be an ideal of $R$ of dimension one and $M$ a finitely generated $R$-module. Then $H_I^i(M)$ is $I$-cofinite for all $i \ge 0$.
\end{lemma}
Now, we are ready to state and prove the first main result of this section, which is an
application of Lemma \ref{L2.2}.
\begin{theorem}\label{T3.2}
Let $R$ be a Noetherian ring, $I$ an ideal of $R$ and $M$ a finitely generated $R$-module.  Then for every $i \geq 0$ and any $j\geq 0$,  the set   $$\{\frak p \in \mathrm{Ass}_R\mathrm{Ext}^j_R(R/I,H^i_I(M))\,\,:\,\, \mathrm{ht}(\frak p/I)\leq 1\}$$ is finite.
\end{theorem}
\begin{proof} Suppose there are $i$ and $j$ such that the set
$$\{ \frak p \in \mathrm{Ass}_R\mathrm{Ext}^j_R(R/I,H^i_I(M)) \,\,:\,\, \mathrm{ht}(\frak p/I) \le 1 \} $$ is not finite. We can choose an countable set $\{ \frak p_k\}_{k \ge 1} \subseteq \mathrm{Ass}_R\mathrm{Ext}^j_R(R/I,H^i_I(M))$ and $\mathrm{ht}(\frak p_k/I) = 1$ for all $k \ge 1$. Let $T$ as Lemma \ref{L2.2}, we have $R \to T$ is a flat extension and
$$\mathrm{Max}(T) = \{\frak p_kT \}_{k \ge 1}.$$
By the flat base change theorem (see, \cite[Theorem 4.3.2]{BS98}) we have
$$\mathrm{Ext}^j_R(R/I,H^i_I(M)) \otimes_R T \cong \mathrm{Ext}^j_T(R/I \otimes_R T,H^i_I(M) \otimes_R T) \cong \mathrm{Ext}^j_T(T/IT,H^i_{IT}(M \otimes_R T)).$$
So $\frak p_kT \in \mathrm{Ass}_{T} \mathrm{Ext}^j_T(T/IT,H^i_{IT}(M \otimes_R T))$ for all $k \ge 1$ by \cite[Theorem 23.2]{M86}. On the other hand we have $\dim T/IT = 1$ so $H^i_{IT}(M \otimes_R T)$ is $IT$-cofinite by Lemma \ref{L3.1}. Thus  the $T$-module $\mathrm{Ext}^j_T(T/IT,H^i_{IT}(M \otimes_R T))$ is finitely generated and so the set $$\mathrm{Ass}_{T} \mathrm{Ext}^j_T(T/IT,H^i_{IT}(M \otimes_R T))$$ is finite, which is  a contradiction. The proof is complete.
\end{proof}
      Recall that $\mathrm{Ass}_RH_I^i(M) = \mathrm{Ass}_R\mathrm{Hom}(R/I,H_I^i(M))$ for all $i \ge 0$. So the following result is an immediately
consequence of Theorem \ref{T3.2}.
\begin{corollary} Let $I$ be an ideal of $R$ and $M$ a finitely generated $R$-module. Then for every $i \ge 0$ the set $\{ \frak p \in \mathrm{Ass}_RH^i_I(M) \,\,:\,\, \mathrm{ht}(\frak p/I) \le 1 \} $ is finite.
  \end{corollary}
The following results are other applications of Lemma \ref{L2.2} to local cohomology
modules.
\begin{corollary}
Let $R$ be a Noetherian ring, $I$ an ideal of $R$ and $n\geq 1$ be an integer and $M$ be a finitely generated $R$-module such that $\dim(M/IM)=n$.  Then for any finitely generated $R$-module $N$ with support in $V(I+\mathrm{Ann}_R(M))$ and for any $i, j \ge 0$ we have the set
$$\{\frak p\in \mathrm{Ass}_R(\mathrm{Ext}^j_R(N,H^i_I(M)))\,\,:\,\,\dim(R/\frak p)\geq n-1\}$$ is finite.
\end{corollary}
\begin{proof} Let $J = \mathrm{Ann}(M/IM)$. Then, we have $V(J) = V(I+\mathrm{Ann}_R(M))$. It is not difficult to see that $H_I^i(M) \cong H_J^i(M)$ for all $i \ge 0$. We can assume henceforth that $I = \mathrm{Ann}(M/IM)$ and $\dim R/I = n$. Notice that if $K$ is an $I$-cofinite module, then $\mathrm{Ext}^j_R(N, K)$ is finitely generated for all finitely generated $R$-module $N$ with support $V(I)$ (see \cite[Lemma 1]{Ka96}). Now the proof is the same as Theorem \ref{T3.2}.
\end{proof}
\begin{corollary}
Let $R$ be a Noetherian ring, $I$ an ideal of $R$ and $n\geq 1$ be an integer and $M$ be a finitely generated $R$-module such that $\dim(M/IM)=n$.  Then for any finitely generated $R$-module $N$ with support in $V(I+\mathrm{Ann}_R(M))$ and for any $i, j \ge 0$ we have the set
$$\{\frak p\in \mathrm{Ass}_R({\rm Tor}^R_j(N,H^i_I(M)))\,\,:\,\,\dim(R/\frak p)\geq n-1\}$$ is finite.
\end{corollary}
\begin{proof} Use \cite[Theorem 2.1]{Me05}.
\end{proof}
We prove the second main result of this section.
\begin{theorem}\label{T3.6}
Let $R$ be a Noetherian ring, $I$ an ideal of $R$ and $M$ an (not necessarily finitely generated) $R$-module. Then for any integer $t\geq 0$, the set $$\mathcal{S}:=\{ \frak p\in \mathrm{Ass}_RH^t_I(M)\,\,:\,\, \mathrm{ht}(\frak p)=t\} = \{ \frak p \in \mathrm{Supp}(H^t_I(M))\,\,:\,\,\mathrm{ht}(\frak p)=t\}$$ is finite.
\end{theorem}
\begin{proof}
It follows from Grothendieck's Vanishing Theorem, that each element of the set $\{ \frak p \in \mathrm{Supp}(H^t_I(M))\,\,:\,\,\mathrm{ht}(\frak p)=t\}$ is a minimal element of the set $\mathrm{Supp}(H^t_I(M))$ and so is an associated prime ideal of the $R$-module $H^t_I(M)$. Therefore
$$\{ \frak p \in \mathrm{Supp}(H^t_I(M))\,\,:\,\, \mathrm{ht}(\frak p)=t\} \subseteq \mathcal{S} \subseteq \{ \frak p \in \mathrm{Supp}(H^t_I(M))\,\,:\,\,\mathrm{ht}(\frak p)=t\}.$$
Hence $$\mathcal{S}=\{ \frak p \in \mathrm{Supp}(H^t_I(M))\,\,:\,\,\mathrm{ht}(\frak p)=t\}.$$
Let $\frak p $ be an arbitrary element of $\{ \frak p \in \mathrm{Supp}(H^t_I(M))\,\,:\,\,\mathrm{ht}(\frak p)=t\}$ we have $H^t_{IR_{\frak p}}(M_{\frak p}) \neq 0$. Notice that $\dim R_{\frak p} = t$ so by \cite[Exercise 6.1.9]{BS98} we have $H^t_{IR_{\frak p}}(M_{\frak p}) = H^t_{IR_{\frak p}}(R_{\frak p}) \otimes_{R_{\frak p}}M_{\frak p}$. Hence $H^t_{IR_{\frak p}}(R_{\frak p}) \neq 0$. Thus for any $R$-module $M$ we have
  $$\{ \frak p \in \mathrm{Supp}(H^t_I(M))\,\,:\,\,\mathrm{ht}(\frak p)=t\} \subseteq \{ \frak p \in \mathrm{Supp}(H^t_I(R))\,\,:\,\,\mathrm{ht}(\frak p)=t\}.$$
So it is enough to prove the assertion in the case $M = R$. Suppose that $\{ \frak p\in \mathrm{Ass}_RH^t_I(R)\,\,:\,\, \mathrm{ht}(\frak p)=t\}$ is not finite. Then, we can choose a countable infinite subset
$$\{\frak p_i\}_{i\ge 1} \subseteq \{ \frak p \in \mathrm{Ass}_RH^t_I(R)\,\,:\,\, \mathrm{ht}(\frak p) = t\}.$$
Now set $T$ as in Lemma \ref{L2.2}. Then we have $R\rightarrow T$ is a flat extension and $\mathrm{ Max}(T)=\{\frak p_iT\}_{i\ge 1}$. In particular, $T$ is a Noetherian ring of dimension $t$ and $\frak p_iT \in \mathrm{Ass}_TH^t_{IT}(T)$ for all $i \ge 1$. But, in view of \cite[Proposition 5.1]{Me05}, the $T$-module $H^t_{IT}(T)$ is Artinian and hence has finitely many associated primes, which is a contradiction.
The proof is complete.

\end{proof}
Let $T$ be a subset of $\mathrm{Spec}(R)$. We denote
$$T_i = \{\frak p \in T \,:\, \mathrm{ht}(\frak p) = i\}.$$
The following is a direct consequence of Theorem \ref{T3.6}.
\begin{corollary}\label{C3.7}
Let $M$ be an $R$-module of finite dimension and $I$ an ideal of $R$. Then the set
$$\bigcup_{i \ge 0} (\mathrm{Ass}_RH^i_I(M))_i$$
is finite.
\end{corollary}
We close this paper with the following remark.
\begin{remark}\rm It is not known whether the set of minimal associated primes of a local cohomology module is finite. It is equivalent to the question: Is the support of local cohomology closed (see \cite{HKM09})? By Lemma \ref{L2.2} one can assume that the set of minimal associated primes of $H^i_I(M)$ is just $\mathrm{Max}(R)$.
\end{remark}

{\bf Acknowledgements:} The second author is grateful to Mahdi Majidi-Zolbanin for his valuable discussions about Section 2 on \verb"Mathoverflow.net".

\textsc{Faculty of Mathematical Sciences, Department of Mathematics, University of Mohaghegh
Ardabili, 56199-11367, Ardabil, Iran}\\
{\it E-mail address}: bahmanpour.k@gmail.com\\
\textsc{Department of Mathematics, FPT University, 8 Ton That Thuyet Road, Ha Noi, Viet Nam}\\
 {\it E-mail address}: quyph@fpt.edu.vn\\

\end{document}